\documentclass[11pt]{amsart}
\usepackage{ae}
\usepackage[all]{xy}
\usepackage{graphicx,amsfonts,amssymb,amsmath}
\usepackage{amsmath}
\usepackage{pifont}

 \usepackage[backref]{hyperref}
\hypersetup{colorlinks,linkcolor=blue,citecolor=red,}
\usepackage[alphabetic,backrefs,lite]{amsrefs}

\usepackage[OT2,T1]{fontenc}
\DeclareSymbolFont{cyrletters}{OT2}{wncyr}{m}{n}
\DeclareMathSymbol{\sha}{\mathalpha}{cyrletters}{"58}
\input cyracc.def

 \newtheorem{thm}{Theorem}[section]

 \newtheorem{cor}[thm]{Corollary}
 \newtheorem{lem}[thm]{Lemma}
 
 \newtheorem{prop}[thm]{Proposition}
 \theoremstyle{definition}
 \newtheorem{defn}[thm]{Definition}
 \theoremstyle{remark}
 
 \theoremstyle{remark}
 \newtheorem{rem}[thm]{Remark}

 \newcommand{\To}{\longrightarrow}

 \newcommand{\E}{\textup{E}}
 
 \newcommand{\im}{\textup{Im}}
 \renewcommand{\ker}{\textup{Ker}}
 
 \newcommand{\fab}{\textup{f-ab}}
  \newcommand{\et}{\textup{et}}
   \newcommand{\desc}{\textup{desc}}
 \newcommand{\Br}{\textup{Br}}
 \renewcommand{\H}{\textup{H}}

 \newcommand{\Hom}{\textup{Hom}}
 \newcommand{\Mor}{\textup{Mor}}

  \newcommand{\Spec}{\textup{Spec}}
 \newcommand{\BM}{Brauer\textendash Manin\ }

\newcommand{\nosp}{\negthinspace}

 \newcommand{\Ind}{\textup{Ind}}

 \newcommand{\A}{\textbf{A}}
 \newcommand{\Q}{\mathbb{Q}}
 \newcommand{\R}{\textup{R}}
  
 \newcommand{\Z}{\mathbb{Z}}

\numberwithin{equation}{section}

\begin{document}

\title[]
 {\'Etale Brauer-Manin obstruction for Weil restrictions}

\author{Yang Cao}
\author{Yongqi Liang}

\address{Yang CAO
\newline University of Scinece and Technology of China,
\newline School of Mathematical Sciences,
\newline 96 Jinzhai Road,
\newline  230026 Hefei, Anhui, China
 }

\email{yangcao1988@ustc.edu.cn}

\address{Yongqi LIANG
\newline University of Scinece and Technology of China,
\newline School of Mathematical Sciences,
\newline 96 Jinzhai Road,
\newline  230026 Hefei, Anhui, China
 }

\email{yqliang@ustc.edu.cn}

\keywords{Weil restriction, \'etale \BM obstruction}
\thanks{\textit{MSC 2020} : 14G12 11G35 14G05p}


\maketitle

\begin{abstract}
For any quasi-projective algebraic variety $X$ defined over a finite extension $L$ of a number field $K$, we prove that the \'etale Brauer\textendash Manin set of $X$ agrees with the \'etale Brauer\textendash Manin set of its Weil  restriction of scalars $\R_{L/K}X$ with respect to $L/K$.
\end{abstract}


\section{Introduction}

To study the set of rational points of an algebraic variety $X$ defined over a number field $L$, we embed it diagonally into the set of ad\'elic points $X(\A_L)$.
A most important tool is various cohomological obstructions applied to $X(\A_L)$ consisting of  two types in principle
\begin{itemize}
\item the \emph{\BM obstruction} defined  by Yu. I. Manin via the Brauer group $\Br(X)=\H^2(X,\mathbb{G}_\textup{m})$;
\item the \emph{descent obstruction} given by torsors under (not necessarily commutative nor connected) linear algebraic groups $G$ which relates to $\H^1(X,G)$;
\end{itemize}
and a mix of them. See \S \ref{recall} for a brief recall.

For an arbitrary subfield $K$ of $L$, when $X$ is  quasi-projective its \emph{Weil restriction} $\R_{L/K}X$ is the algebraic variety defined over $K$ such that for any  $K$-scheme $S$
$$\Mor_K(S,\R_{L/K}X)=\Mor_L(S_L,X),$$
where $S_L=S\times_{\Spec(K)}\Spec(L)$. In particular, we have a natural identification
$$\Phi: X(\A_L)\buildrel{=}\over\longrightarrow(\R_{L/K}X)(\A_K),$$
under which the subset $X(L)$ is identified with $\left(\R_{L/K}X\right)(K)$.

In \cite[Remark 5, page 95]{CTPoonen00}, J.-L. Colliot-Th\'el\`ene and B. Poonen ask whether the existences of \BM obstruction to Hasse principle on $X$ and on $\R_{L/K}X$ are equivalent to each other. We cannot help pushing it even further to ask whether surviving subsets of diverse cohomological obstructions are identified under $\Phi$. For the \emph{finite} descent obstruction (and similarly for the \emph{finite solvable} or \emph{finite abelian} descent obstructions), it has been affirmatively answered by M. Stoll in \cite[Proposition 5.15]{Stoll07}.

In the present paper, we prove the following main result.
\begin{thm}\label{mainthm}
Let $L/K$ be an extension of number fields. Let $X$ be a smooth quasi-projective geometrically connected variety defined over $L$. Denote by $X'=\R_{L/K}X$ its Weil restriction.

Then the \'etale \BM subsets of $X$ and of $X'$ coincide under the identification $\Phi: X(\A_L)\buildrel{=}\over\longrightarrow X'(\A_K)$, i.e.
\begin{equation}\label{maineq}
\Phi\left(X(\A_L)^{\et,\Br}\right)=X'(\A_K)^{\et,\Br}.
\end{equation}
\end{thm}

As a by-product, from the  inclusion $\supset$ of the equality (\ref{maineq}), which is the easy half, we deduce the following corollary.
\begin{cor}[Corollary \ref{extiongroundfield}]\label{maincor}
Let $L/K$ be an extension of number fields. Let $V$ be a smooth quasi-projective geometrically connected variety defined over $K$.
Then
$$V(\A_K)^{\et,\Br}\subset V_L(\A_L)^{\et,\Br},$$
in particular, the existence of the \'etale \BM obstruction to Hasse principle on $V_L$ implies the existence of the same obstruction on $V$.
\end{cor}
The analogous statement of Corollary \ref{maincor} for the \BM obstruction
holds as well and it can  be deduced from the following inclusion of \BM subsets (see Corollary \ref{easyinclusion} for the proof of this inclusion)
\begin{equation}\label{reverseBM}
\Phi\left(X(\A_L)^{\Br}\right)\supset X'(\A_K)^{\Br}
\end{equation}
in place of (\ref{maineq}).
An alternative approach of this analogy was suggested by O. Wittenberg and presented by B. Creutz and B. Viray in \cite[Lemma 2.1(1)]{CreutzViray21}.

However, the other inclusion for \BM sets
\begin{equation}\label{eqBr}
\Phi\left(X(\A_L)^{\Br}\right)\subset(\R_{L/K}X)(\A_K)^{\Br}
\end{equation}
remains open.

The paper is organized as follows. First of all, we recall definitions and several known facts about cohomological obstructions in \S \ref{recall}.
Secondly, in \S \ref{easyhalf} we establish (\ref{reverseBM}) and the easy inclusion $\supset$ of the equality (\ref{maineq}), and we deduce Corollary \ref{maincor} and its analogy as well. The proof of the other inclusion $\subset$ of (\ref{maineq}) is then divided into two parts, the first focuses on abelian cohomologies and the second concerns non-abelian cohomologies. In the first part \S \ref{unfairsection}, we compare the Brauer groups of $X$ and $\R_{L/K}X$ to prove an ``unfair'' intermediate comparison (\ref{unfaireq}) which is a weaker version of the difficult inclusion of (\ref{maineq}). In the second part \S \ref{mainproofsection}, we make use of the theory of finite torsors developed by M. Stoll to complete the proof of Theorem \ref{mainthm}.

Note that the key intermediate comparison (\ref{unfaireq}) is weaker than (\ref{eqBr}), it does not allow us to prove (\ref{eqBr}).
When we pass ``limit''  to the ``universal cover'', it turns out to be the inclusion $\subset$ of the equality (\ref{maineq}).

\section{Definitions and known facts about cohomological obstructions}\label{recall}

\subsection{Notation and conventions}\label{notationsection}
In this paper, the ground field $k$ is a number field, for which we denote by $\A_k$ its ring of ad\`eles.
Let $\Omega_k$ denote the set of places of $k$.
The absolute Galois group of $k$ is denoted by $\Gamma_k$.
A $k$-variety $V$ means a separated scheme of finite type over $k$. For any field extension $k'$ of $k$, denote by $V_{k'}$ the base extension of $V$ to $k'$. Fix an algebraic closure $\overline{k}$ of $k$, the variety $V_{\overline{k}}$ is also denoted by $\overline{V}$.

Cohomology groups in the paper are \'etale cohomology or Galois cohomology (not necessarily commutative). The Brauer group $\Br(V)$ is the cohomological Brauer group $\H^2(V,\mathbb{G}_\textup{m})$. The structure morphism $V\to\Spec(k)$ induces a natural homomorphism $\H^i(k,\mathcal{F})\to\H^i(V,\mathcal{F})$ for \'etale cohomologies with a fixed sheaf $\mathcal{F}$ (usually $\mu_n$ or $\mathbb{G}_\textup{m}$ in this paper), we will denote their images uniformly by $\im(\H^i(k,-))$ for different varieties $V$ since they fit into any naturally induced commutative diagrams, though these images could be different as abstract groups when $V$ varies.

In this paper, when we state general assertions, we will use the notation above (e.g. $V$, $k$, etc.). While when we  specify to our situation concerning Weil restrictions, we will turn to the notation below (e.g. $X$, $X'$, $L$, $K$, etc.).

Let $L/K$ be an extension of number fields. From now on, we fix an algebraic closure $\overline{L}$ of $L$ (and $K$).
For a quasi-projective $L$-variety $X$, its Weil restriction $\R_{L/K}X$ exists according to \cite[Theorem 4 \S 7.6]{Neronmodel}. It will be denoted by $X'=\R_{L/K}X$. In the whole paper, the $L$-variety $X$ is assumed to be smooth and geometrically connected, then $X'$ is a smooth geometrically connected quasi-projective $K$-variety.

By Weil restriction, the identity $\textup{id}:X'\to X'$ induces a morphism $X'_L\to X$ giving rise to a homomorphism $\H^i(X,\mathcal{F})\to\H^i(X'_L,\mathcal{F})$ for the abelian sheaf $\mathcal{F}=\mu_n$ or $\mathbb{G}_\textup{m}$. By composing with the corestriction homomorphism $\textup{Cores}_{L/K}:\H^i(X'_L,\mathcal{F})\to\H^i(X',\mathcal{F})$, we obtain a homomorphism $\phi:\H^i(X,\mathcal{F})\to\H^i(X',\mathcal{F})$.
Similar constructions also exist between each term that appears in cohomological  spectral sequences of $X$ and $X'$. We always use $\phi$ (with a subscript if necessary) to denote these natural homomorphisms from cohomologies of $X$ to cohomologies of $X'$.


\subsection{Torsors}\label{torsorrecall}
We refer to the book \cite{Skbook}  by A. Skorobogatov for more details on torsors. We recall some facts that will be used in this paper.

Let $H$ be a linear algebraic group over $k$. Let $h:W\to V$ be a $V$-torsor under $H$ (with group action on the right), also denoted by $h:W\buildrel{H}\over\to V$. Up to an isomorphism of torsors, it represents a unique class $[W]$ in the cohomology set (or cohomology group if $H$ is commutative) $\H^1(V,H)$.

For a class of a $1$-cocycle $[\sigma]\in\H^1(k,H)$, we obtain a twisted $V$-torsor ${_\sigma\nosp h}:{_\sigma\nosp W}\to V$ under the twisted $k$-group $_\sigma\nosp H$. This gives rise to a bijection between the pointed sets $\H^1(V,H)$ and $\H^1(V,{_\sigma\nosp H})$ mapping $[W]$ to $[{_\sigma\nosp W}]$.

We observe that the torsor $W$ becomes trivial when being base extended by $W\to V$ since the diagonal is a section of $W\times_V W\to W$. In other words, the class $[W]$ maps to the trivial class in $\H^1(W,H)$.

Let us turn to Weil restrictions.
Let $G$ be a linear algebraic group over $L$ and $f:Y\to X$ be a torsor under $G$. Consider the linear $K$-group $G'=\R_{L/K}G$
and the $K$-variety $Y'=\R_{L/K}Y$. According to \cite[Corollary A.5.4(3)]{pseudoreductivegroups}, the induced morphism $f':Y'\to X'$ is a $X'$-torsor under $G'$.

\subsection{Descent obstruction}\label{descentobsrecall}
For a $V$-torsor $h:W\buildrel{H}\over\to V$ under a linear algebraic group $H$ over $k$, the subset $V(\A_k)^h$ of ad\'elic points that survive the torsor is defined to be
$$V(\A_k)^h=\bigcup_{[\sigma]\in\H^1(k,H)}{_\sigma\nosp h}({_\sigma\nosp W}(\A_k)).$$
The descent subset is defined to be
$$V(\A_k)^\desc=\bigcap_{h:W\buildrel{H}\over\to V}V(\A_k)^h,$$
where the intersection is taken over all $V$-torsors $h:W\buildrel{H}\over\to V$ under all linear algebraic groups $H$ over $k$.
In the same way, we define respectively $V(\A_k)^\textup{conn}$, $V(\A_k)^\et$, and $V(\A_k)^\fab$, by considering respectively connected linear algebraic groups, finite algebraic groups, and finite abelian algebraic groups.

\subsection{Brauer\textendash Manin obstruction}\label{BMobsrecall}
In \cite{Manin}, Yu. I. Manin defined a pairing for a $k$-variety $V$
\begin{equation}\label{BMpairing}\left(\quad,\quad\right)_{\textup{BM}_k}:~V(\A_k)\times\Br(V)\to\Q/\Z\end{equation}
under which  the diagonally embedded set of rational points $V(k)\subset V(\A_k)$ vanishes and the image of $\Br(k)\to\Br(V)$ does not affect. The \BM subset of ad\'elic points annihilated by all elements of $\Br(V)$ is denoted by $V(\A_k)^\Br$.

Consider the $L$-variety $X$ and its Weil restriction $X'$ over $K$.
As argued in \cite[Lemma 2.1(1)]{CreutzViray21}, because  the corestriction map is compatible with pullback \cite[Proposition 3.8.1]{Brauergroup}, we obtain the following compatibility of the \BM pairing
$$\left((x_w)_{w\in\Omega_L},b\right)_{\textup{BM}_L}=\left(\Phi((x_w)_{w\in\Omega_L}),\phi_\Br(b)\right)_{\textup{BM}_K}$$
for any $b\in\Br(X)$ and any $(x_w)_{w\in\Omega_L}\in X(\A_L)$.

\subsection{\'Etale \BM obstruction}\label{etBMobsrecall}
The \BM pairing applied to \'etale covers of the $k$-variety $V$ defines  the \'etale \BM subset
$$V(\A_k)^{\et,\Br}=\bigcap_{h:W\buildrel{H}\over\to V}V(\A_k)^{h,\Br}$$
with
$$V(\A_k)^{h,\Br}=\bigcup_{[\sigma]\in\H^1(k,H)}{_\sigma\nosp h}({_\sigma\nosp W}(\A_k)^\Br),$$
where the intersection is taken over all $V$-torsors $h:W\buildrel{H}\over\to V$ under all \emph{finite} $k$-groups $H$.
When we restrict ourselves to torsors under \emph{finite abelian} $k$-groups $H$, we define similarly the finite abelian \BM subset $V(\A_k)^{\fab,\Br}$.

\subsection{Comparisons}\label{comparisonrecall}
Some obstructions recalled above are comparable.
When the $k$-variety $V$ is  smooth and geometrically connected, the work of D. Harari \cite[Th\'eor\`eme 2 and Remarque 4]{Harari02} showed that $$V(\textbf{A}_k)^{\textup{conn}}=V(\textbf{A}_k)^{\Br}.$$ The work of C. Demarche \cite{Demarche09}, A. Skorobogatov \cite{Sk09etBMdesc}, generalised to quasi-projective varieties by the first  author, C. Demarche, and F. Xu \cite[Theorem 7.5]{CDX}, showed that $$V(\A_k)^\desc=V(\A_k)^{\et,\Br}.$$

\section{The easy inclusion and an application}\label{easyhalf}

From now on, we fix  an  extension of number fields $L/K$. We consider  a smooth  geometrically connected quasi-projective variety $X$ over $L$ and its Weil restriction $X'=\R_{L/K}X$ over $K$.

\begin{prop}\label{fdescentformula}
Let $G$ be a linear algebraic group over $L$. Consider a torsor $f:Y\buildrel{G}\over\to X$ under $G$ over $L$ and its Weil restriction $f':Y'\buildrel{G'}\over\to X'$ over $K$.

Under the identification $\Phi:X(\A_L)\to X'(\A_K)$  we have an equality
$$\Phi\left(X(\textbf{A}_L)^f\right)=X'(\textbf{A}_K)^{f'}.$$
\end{prop}

\begin{proof}
By Shapiro's lemma \cite[\S I.5.8.b)]{SerreGaloisCoh}, we identify the pointed sets $\H^1(L,G)$ and $\H^1(K,G')$ by mapping a class of $L$-torsor $[P]$ to the class of $K$-torsor $[\R_{L/K}P]$, being both denoted by $[\sigma]$ for a certain $1$-cocycle $\sigma$. The Weil restriction of the twisted torsor ${_\sigma\nosp f}:{_\sigma\nosp Y}\to X$ under $_\sigma\nosp G$ is nothing but the torsor ${_\sigma\nosp f}':{_\sigma\nosp Y}'\to X'$ under $_\sigma\nosp G'$. Whence $\Phi({_\sigma\nosp Y}(\A_L))={_\sigma\nosp Y}'(\A_K)$ and we obtain
\begin{equation*}
    \begin{array}{ll}
        \displaystyle\Phi\left(X(\textbf{A}_L)^f\right) &\displaystyle=\Phi\left(\bigcup_{[\sigma]\in\H^1(L,G)} {_{\sigma}\nosp f}({_{\sigma}\nosp Y}(\textbf{A}_L))\right)\\
        &\\
                  &\displaystyle = \bigcup_{[\sigma]\in\H^1(K,G')}{_{\sigma}\nosp f}'({_{\sigma}\nosp Y}'(\textbf{A}_K))\\
        &\\
                  &\displaystyle =X'(\textbf{A}_K)^{f'}.
    \end{array}
\end{equation*}
\end{proof}

\begin{cor}\label{easyinclusion}
Let $L/K$ be an  extension of number fields and let $X$ be a smooth  geometrically connected quasi-projective variety over $L$. Denote by $X'$ its Weil restriction over $K$.

Then under the identification $\Phi:X(\A_L)\to X'(\A_K)$  we have the following inclusions:
\begin{itemize}
\item[] $ \Phi\left(X(\textbf{A}_L)^{\textup{desc}} \right)\supset(\textup{R}_{L/K}X)(\textbf{A}_K)^{\textup{desc}},$
\item[]
 $\Phi\left(X(\textbf{A}_L)^{\textup{conn}}\right)\supset(\textup{R}_{L/K}X)({\textbf{A}}_K)^{\textup{conn}},$
\end{itemize}
or equivalently
\begin{itemize}
\item[] $ \Phi\left(X({\textbf{A}}_L)^{\textup{et},\Br}\right)\supset(\textup{R}_{L/K}X)({\textbf{A}}_K)^{\textup{et},\Br},$
\item[] $\Phi\left( X({\textbf{A}}_L)^{\Br}\right)\supset(\textup{R}_{L/K}X)({\textbf{A}}_K)^{\Br}.$
\end{itemize}
\end{cor}

\begin{proof}
Let the notation $f:Y\buildrel{G}\over\to X$ and $f':Y'\buildrel{G'}\over\to X'$ be as above.

By taking intersection over all torsors $f:Y\to X$  under all linear algebraic groups $G$ over $L$,  the formula in Proposition \ref{fdescentformula} implies that
$$\Phi\left(X({\textbf{A}}_L)^{\textup{desc}}\right)=\bigcap_{f:Y\to X} X'({\textbf{A}}_K)^{f'}\supset X'({\textbf{A}}_K)^{\textup{desc}}$$ as a subset of $X'({\textbf{A}}_K)$.
(Since not all linear algebraic groups over $K$ are of the form $\R_{L/K}G$ for a certain $G$, the converse inclusion is not clear at this moment.)

When $G$ is connected, the group $G'=\textup{R}_{L/K}G$ is also connected, cf. \cite[Proposition A.5.9]{pseudoreductivegroups}. It follows from  the intersection taken only for connected linear algebraic groups that $$\Phi\left(X({\textbf{A}}_L)^{\textup{conn}}\right)\supset X'({\textbf{A}}_K)^{\textup{conn}}$$ as well.

The last two assertions follow from the comparisons of obstructions: \S \ref{comparisonrecall} applied to both $X$ and $X'$.
\end{proof}

\begin{rem}
For the \'etale \BM subset, the converse inclusion is also valid but  the proof is much more sophisticated involving the rest of the paper.
\end{rem}

\begin{rem}
The last inclusion (for the \BM obstruction) can also be deduced from the compatibility (\ref{BMpairing}) of the \BM pairing recalled in \S \ref{BMobsrecall}. This alternative argument dates back to Colliot-Th\'el\`ene and Poonen \cite[Remark 5, page 95]{CTPoonen00}.
\end{rem}

As an immediate consequence, we have the following application to the behavior of the \'etale \BM obstruction under extension of the ground field.
An another proof of the second statement was suggested by O. Wittenberg and presented recently by B. Creutz and B. Viray in \cite[Lemma 2.1(1)]{CreutzViray21}.

\begin{cor}\label{extiongroundfield}
Let $L/K$ be an extension of number fields. Let $V$ be a smooth quasi-projective geometrically integral variety defined over $K$. Then
\begin{itemize}
\item[(1)] $V({\textbf{A}}_K)^{\textup{et},\Br}\subset V_L({\textbf{A}}_L)^{\textup{et},\Br}$,
\item[(2)] $V({\textbf{A}}_K)^{\Br}\subset V_L({\textbf{A}}_L)^{\Br}$.
\end{itemize}
\end{cor}

\begin{proof}
The identity $V_L\to V_L$ induces a $K$-morphism $V\to \textup{R}_{L/K}(V_L)$. The composition of the corresponding map $V({\textbf{A}}_K)\to (\textup{R}_{L/K}(V_L))({\textbf{A}}_K)$  with $\Phi^{-1}:(\R_{L/K}(V_L))(\A_K)\to V_L(\A_L)$ is nothing but the natural inclusion $V(\A_K)\to V_L(\A_L)$. The desired inclusions then follow from Theorem \ref{easyinclusion} applied to the $L$-variety $X=V_L$ and the functoriality of obstructions applied to $V\to \textup{R}_{L/K}(V_L)$.
\end{proof}

\section{An unfair comparison}\label{unfairsection}

In this section, we prove the following comparison which looks unfair: a certain kind of \'etale \BM obstruction of $X$ and the \BM obstruction of $X'$ are compared. However, it turns out to be the crucial intermediate step which leads to our main result.

\begin{prop}\label{unfairprop}
Let $L/K$ be an extension of number fields. Let $X$ be a smooth quasi-projective geometrically connected variety defined over $L$. Denote by $X'$ its Weil restriction over $K$.

Then under the identification $\Phi: X(\A_L)\longrightarrow X'(\A_K)$ we have an inclusion
\begin{equation}\label{unfaireq}
\Phi\left(X(\A_L)^{\fab,\Br}\right)\subset X'(\A_K)^{\Br}.
\end{equation}
\end{prop}

\subsection*{First reduction}
Since the Brauer group of a smooth variety $V$ is torsion, it is the union of its $n$-torsion parts for all positive integers $n$.
By the Kummer sequence, we have a surjective homomorphism
$$\H^2(V,\mu_n)\To\Br(V)[n]=\textup{Ker}[\Br(V)\buildrel{\cdot n}\over\to\Br(V)].$$
We denote by $V(\A_k)^{\H^2(V,\mu_n)}$ the \BM subset given by the subgroup $\Br(V)[n]$. It suffices to prove for every positive integer $n$ the following inclusion
$$\Phi\left(X(\A_L)^{\fab,\Br}\right)\subset X'(\A_K)^{\H^2(X',\mu_n)}.\leqno{(\ref{unfaireq}')}$$

\subsection*{Convention}
As most cohomology groups in this section  have coefficients in the sheaf $\mu_n$, to simplify the notation we write $\H^i(V)$ in place of $\H^i(V,\mu_n)$ for any variety $V$.

\subsection*{Structure of the proof}
The proof of $(\ref{unfaireq}')$ is rather long, it consists of several main ingredients.
Spectral sequences allows us to compare Brauer groups of $X$ and $X'$. Over an algebraic closure of concerned number fields, second degree cohomologies  of products were studied by Skorobogatov\textendash Zarhin, and were extended by the first author to quasi-projective varieties via the theory of universal torsors of $n$-torsion. First of all, we recall relevant notions in \S \ref{universaltorsorsection}-\ref{ss-subsection}. Secondly, we study cohomologies of the Weil restriction in \S \ref{weilressubsection} and compare diverse cohomology groups in \S \ref{comparisonsubsection}. We complete the proof in \S \ref{proofofunfairsubsection}.

\subsection{Universal torsors of $n$-torsion}\label{universaltorsorsection}
Let $V$ be a smooth geometrically connected  variety defined over a number field $k$. D. Harari and A. Skorobogatov generalised the theory of torsors to not necessarily proper varieties and defined the extended type of a $V$-torsor. They proved that a $V$-torsor of a given extended type exists if $V(\A_k)^\Br\neq\varnothing$, cf. \cite[Corollary 3.6]{HSk13}.

In particular, for  a positive integer $n$ the first  author defined a \emph{universal torsor  of $n$-torsion} $h:W\buildrel{H}\over\to V$ to be a $V$-torsor under a finite commutative $k$-group $H$ of $n$-torsion such that its extended type induces an isomorphism of Galois modules $H^*\buildrel{\simeq}\over\To\H^1(\overline{V})$, where $H^*=\Hom(H,\mu_n)$ denotes the group of characters of $H$, cf. D\'efinition 2.1, Proposition 2.2 and the equation (2-3) thereafter in \cite{Cao2017}.
Universal torsors of $n$-torsion are again smooth and geometrically connected, cf. the paragraph immediately before \cite[Corollaire 2.4]{Cao2017}, to which we are allowed to apply the forthcoming Proposition \ref{H12} and Lemma \ref{lem2}.

To prove Proposition \ref{unfairprop}, we may assume that $X(\A_L)^{\fab,\Br}\neq\varnothing$. A fortiori $X(\A_L)^{\Br}\neq\varnothing$ and therefore a universal $X$-torsor of $n$-torsion $f:T\buildrel{S}\over\to X$ exists for every fixed $n$. By passing to $\overline{L}$ with the help of the forthcoming Proposition \ref{H12}(\ref{H1}), we check that its Weil restriction $f':T'\buildrel{S'}\over\to X'$ is  a universal $X'$-torsor of $n$-torsion. For any class $[\sigma]\in\H^1(L,S)=\H^1(K,S')$, the twist ${_\sigma\nosp f}:{_\sigma\nosp T}\buildrel{_\sigma\nosp S}\over\to X$ is also a universal torsor of $n$-torsion whose Weil restriction ${_\sigma\nosp f'}:{_\sigma\nosp T'}\buildrel{_\sigma\nosp S'}\over\to X'$ is the twist of $f'$ by $\sigma$.

\subsection{Cohomologies of products}
The first author generalised results of A. Skorobogatov and Yu. Zarhin \cite{SkZarhin} to quasi-projective varieties.
The following result is the special case for $\mu_n$ of \cite[Corollaire 2.7]{Cao2017} (see also its corrigendum \cite[Th\'eor\`eme 2.1]{CaoCorr2017}).

Let $\overline{V_1}$ and $\overline{V_2}$ be  smooth  geometrically connected variety defined over a separably closed field $C$ (which will be $\overline{L}$ for our application). Denote by $p_1$ and $p_2$ the canonical projections from the product $\overline{V_1}\times\overline{V_2}$ to its components. As $\overline{V_1}$ and $\overline{V_2}$ possess $C$-rational points, there exists a universal torsor of $n$-torsion over each of them  respectively denoted by $\overline{W_1}\buildrel{\overline{S_1}}\over\to\overline{V_1}$ and $\overline{W_2}\buildrel{\overline{S_2}}\over\to\overline{V_2}$, cf. \cite[Proposition 1.3]{HSk13}.
We define $$\epsilon_{12}:\Hom(\overline{S_1}\otimes \overline{S_2},\mu_n)\to\H^2(\overline{V_1}\times\overline{V_2})$$
to be
$$\psi\mapsto\psi_*(p^*_1[\overline{W_1}]\cup p^*_2[\overline{W_2}]).$$ Here, and from now on, the tensor product $\otimes$ is always taken over $\Z/n\Z$ which is omitted to simplify the notation.

\begin{prop}\label{H12}
Let  $\overline{V_1}$ and $\overline{V_2}$ be  smooth  geometrically connected variety defined over a separably closed field.
There are  isomorphisms of abelian groups
\begin{enumerate}
\item\label{H1} $\xymatrix{
\H^1(\overline{V_1})\oplus\H^1(\overline{V_2})\ar[r]^-{p^*_1+p^*_2}&\H^1(\overline{V_1}\times\overline{V_2}),
}$
\item\label{H2} $\xymatrix{
\H^2(\overline{V_1})\oplus\H^2(\overline{V_2})\oplus\Hom(\overline{S_1}\otimes\overline{S_2},\mu_n)\ar[rr]^-{p^*_1+p^*_2+\epsilon_{12}}&&\H^2(\overline{V_1}\times\overline{V_2}).
}$
\end{enumerate}
\end{prop}

\subsection{Leray spectral sequence}\label{ss-subsection}

Consider a first quadrant spectral sequence $$\E_2^{i,j}\Longrightarrow \H^{i+j}$$ that converges to $\H^{\bullet}$. In particular, $\H^2$ has a filtration $$\H^2=\H^2_0\supset\H^2_1\supset\H^2_2\supset\H^2_3=0$$ such that $\H^2_i/\H^2_{i+1}\simeq\E_\infty^{i,2-i}$ for $i\in\{0,1,2\}$. In other words, $$\H_1^2=\ker(\H^2\to\E_\infty^{0,2})=\ker(\H^2\to\E_2^{0,2})$$ since $\E_\infty^{0,2}\to\E_2^{0,2}$ is injective.

\begin{prop}\label{ssprop}
Let $\E_2^{i,j}\Longrightarrow \H^{i+j}$ be a convergent first quadrant spectral sequence.
Suppose that the homomorphism $\E_2^{3,0}\to\H^3$ is injective, then the following sequences are exact
$$\E_2^{2,0}\to\H_1^2\to\E_2^{1,1}\to0,\leqno{(1)}$$
$$0\to\H_1^2\to\H^2\to\E_2^{0,2}\to\E_2^{2,1}.\leqno{(2)}$$
\end{prop}

\begin{proof}
For a first quadrant spectral sequence, we have $\E_4^{3,0}=\E_\infty^{3,0}\subset\H^3$ and successive surjective homomorphisms
$$\E_2^{3,0}\to\E_3^{3,0}\to\E_4^{3,0}.$$ By assumption, the composition $\E_2^{3,0}\to\H^3$ is injective, then both $\E_2^{3,0}\to\E_3^{3,0}$ and $\E_3^{3,0}\to\E_4^{3,0}$ are  isomorphisms.

(1) Consider the exact sequence
$$0\to\E_{\infty}^{2,0}\to \H_1^2\to\E_{\infty}^{1,1}\to0$$
deduced from the filtration of $\H^2$. As $\E_2^{2,0}$ maps surjectively onto $\E_{\infty}^{2,0}$,
$$\E_2^{2,0}\to \H_1^2\to\E_{\infty}^{1,1}\to0$$ is exact. On the other hand, the exact sequence given by definition
$$0\to\E_3^{1,1}\to\E_2^{1,1}\to\E_2^{3,0}\buildrel{\simeq}\over\to\E_3^{3,0}\to0$$
with the last map  an isomorphism implies that $\E_{\infty}^{1,1}=\E_3^{1,1}=\E_2^{1,1}$, then the first desired exact sequence follows.

(2) Again from the filtration of $\H^2$, we have an exact sequence
$$0\to\H_1^2\to\H^2\to\E_\infty^{0,2}\to0.$$
On the other hand, we know that $\E_\infty^{0,2}=\E_4^{0,2}$ and we have an exact sequence by definition
$$0\to\E_4^{0,2}\to\E_3^{0,2}\to\E_3^{3,0}\buildrel{\simeq}\over\to\E_4^{3,0}\to0$$
with the last map an isomorphism. It follows that $$\E_\infty^{0,2}=\E_3^{0,2}=\ker(\E_2^{0,2}\to\E_2^{2,1}).$$ Whence we obtain the second desired  exact sequence.
\end{proof}

For a $k$-variety $V$, we have the Leray spectral sequence for $V\to\Spec(k)$ with coefficients in the sheaf $\mu_n$
\begin{equation}\label{ss}
\E_2^{i,j}=\H^i(k,\H^j(\overline{V}))\Longrightarrow \H^{i+j}(V).
\end{equation}
The term $\H_1^2$ appeared in the filtration of $\H^2=\H^2(V)$ is the kernel $\ker(\H^{2}(V)\to\H^2(\overline{V})).$

When $V$ is geometrically connected, we have $\H^0(\overline{V})=\mu_n$ and hence $\E_2^{i,0}=\H^i(k)$ for all $i$. The homomorphism $\E_2^{3,0}\to\H^3$ in Proposition \ref{ssprop} is nothing but $\H^3(k)\to\H^3(V)$.

\begin{lem}\label{H3}
Suppose that $k$ is a number field and  $V(\A_k)\neq\varnothing$, then the homomorphism
$$\H^3(k)\to\H^3(V)$$
is injective.
\end{lem}
\begin{proof}
Under the assumption, the variety $V$ possesses  $k_v$-rational points for each $v\in\Omega_k$ and hence $\H^3(k_v)\to\H^3(V_{k_v})$ is injective. According to \cite[Ch. II \S 6 Th\'eor\`eme B]{SerreGaloisCoh} applied to the finite sheaf $\mu_n$, the homomorphism $\H^3(k)\to\prod_{v\in\Omega_k}\H^3(V_{k_v})$ is an isomorphism. Then the desired injectivity follows from the commutative diagram
$$\xymatrix{
\H^3(k)\ar[r]\ar[d]&\H^3(V)\ar[d]\\
\prod_{v\in\Omega_k}\H^3(k_v)\ar[r]&\prod_{v\in\Omega_k}\H^3(V_{k_v})
}$$
\end{proof}

\begin{cor}\label{sscor}
Suppose that $V$ is a geometrically connected variety defined over a number field $k$ such that $V(\A_k)\neq\varnothing$. If $\H_1^2(V)$ denotes the kernel  $\ker(\H^{2}(V)\to\H^2(\overline{V})),$ then we have exact sequences
$$\displaystyle\H^2(k)\to\H_1^2(V)\to\H^1(k,\H^1(\overline{V}))\to0,\leqno{(1)}$$
$$\displaystyle0\to\H_1^2(V)\to\H^2(V)\to\H^0(k,\H^2(\overline{V}))\to\H^2(k,\H^1(\overline{V})).\leqno{(2)}$$
\end{cor}
\begin{proof}
It is an immediate consequence of Proposition \ref{ssprop} and Lemma \ref{H3}.
\end{proof}

\subsection{Cohomologies of Weil restrictions}\label{weilressubsection}
Denote by $d$ the degree of the extension $L/K$.
For the fixed algebraic closure $\overline{K}$ of $K$, each $K$-embedding $\iota_{s}:L\to \overline{K}(1\leq s\leq d)$ extends to an element (also denoted by) $\iota_s:\overline{K}\to\overline{K}$ of $\Gamma_K$.
Let $\Gamma_K/\Gamma_L=\{\iota_s\Gamma_L|1\leq s\leq d\}$ be the set of left cosets of $\Gamma_L$ in $\Gamma_K$, and we denote by $\gamma_s$ the coset $\iota_s\Gamma_L$ for short.  If $\gamma\in\Gamma_K$ satisfies $\iota_t^{-1}\gamma\iota_s\in\Gamma_L$ for some $s$ and $t$, we define $\gamma_t\cdot\gamma=\gamma_s$. Then this is a right action of $\Gamma_K$ on $\Gamma_K/\Gamma_L$. For each embedding $\iota_{s}:L\to \overline{K}$, we define $\overline{X}\nosp{\gamma_s}=X\otimes_{L,\iota_s}\overline{K}$ to be the base change of $X$ from $L$ to $\overline{K}$ via $\iota_s$. We can check that the $\overline{K}$-variety $\overline{X}\nosp{\gamma_s}$ does not depend on the choice of representative $\iota_s$ of $\gamma_s$, it depends only on the coset $\gamma_s$ as indicated by the notation.
Since Weil restriction is compatible with base change,
 we have $\displaystyle\overline{X'}=\prod_{s=1}^d \overline{X}\nosp{\gamma_s}$.
We remark that the schemes $\overline{X}\nosp{\gamma_s}(1\leq s\leq d)$ are isomorphic to each other (not as $\overline{K}$-varieties). More precisely, if $\gamma_t\cdot\gamma=\gamma_s$ i.e. $\iota_t^{-1}\gamma\iota_s\in\Gamma_L$, the left Galois action of $\gamma$ on $\overline{K}$ (as the coefficient field of equations defining varieties $\overline{X}\nosp_s$) gives rise to an isomorphism of schemes $\gamma:\overline{X}\nosp{\gamma_t}\to\overline{X}\nosp{\gamma_s}$ such that the following diagram commutes
$$\xymatrix{
\overline{X}\nosp{\gamma_t}\ar[rr]^\gamma\ar[d]&&\overline{X}\nosp{\gamma_s}\ar[d] \\
\Spec(\overline{K})\ar[rd]_{\iota^*_t}\ar[rr]^{\gamma^*}&&\Spec(\overline{K})\ar[ld]^{\iota_s^*}\\
&\Spec(L)&
}$$
where $\gamma^*:\Spec(\overline{K})\to\Spec(\overline{K})$ is induced by $\gamma\in\Gamma_K$ and the triangle commutes since $(\iota_t^{-1}\gamma\iota_s)_{|L}=\textup{id}_L$.
Hence $\Gamma_K$ acts (on the right) on the product $\displaystyle\overline{X'}=\prod_{s=1}^d \overline{X}\nosp{\gamma_s}$ permuting the components as indicated by the notation itself. 
To simplify the notation, we write $\overline{X}\nosp_{s}=\overline{X}\nosp{\gamma_s}$.
Without lost of generality, we may assume that $\gamma_1$ is the subgroup $\Gamma_L$ and we simply write $\overline{X}$ for ${\overline{X}\nosp_1}$.
Given  its natural left  action on $\overline{K}$, there is an associated  left action of the Galois group $\Gamma_L$ on $\H^i(\overline{X})$. This last action gives rise to the left action of $\Gamma_K$ on  $\displaystyle\bigoplus_{s=1}^d\H^i(\overline{X}\nosp_s)$ by permutation. In other words, as a Galois module $\displaystyle\bigoplus_{s=1}^d\H^i(\overline{X}\nosp_s)$ is nothing but the induced module $\displaystyle\Ind_{\Gamma_L}^{\Gamma_K}\H^i(\overline{X})$.

If $p_s:\overline{X'}\to\overline{X}\nosp_s$ denotes the $s$-th natural projection, we claim that the associated homomorphism of abelian groups
$$\xymatrix{\displaystyle\bigoplus_{s=1}^d\H^i(\overline{X}\nosp_s)\ar[r]^-{\sum{p^*_s}}&\displaystyle\H^i(\prod_{s=1}^d\overline{X}\nosp_s)=\H^i(\overline{X'})}$$
is moreover a homomorphism of  $\Gamma_K$-modules.
Indeed,  if $\gamma_t\cdot\gamma=\gamma_s$ i.e. $\iota_t^{-1}\gamma\iota_s\in\Gamma_L$, then the right action of $\gamma$ on varieties sends $\overline{X}\nosp_t$ to $\overline{X}\nosp_s$. The associated left action on the cohomology level  denoted by $\gamma_\H=\gamma^*$  sends $\H^i(\overline{X}\nosp_s)$ to  $\H^i(\overline{X}\nosp_t)$. It follows from $\gamma\circ p_t=p_s\circ \gamma$ that $p^*_t\circ\gamma_\H=\gamma_\H\circ p^*_s$, and therefore the homomorphism $\sum{p^*_s}$ preserves the actions of any $\gamma\in\Gamma_K$.

\begin{prop}\label{inducedmodule}\ \\
\begin{enumerate}
\item For $i=1$,
$$\xymatrix{
\displaystyle\bigoplus_{s=1}^d\H^1(\overline{X}\nosp_s)\ar[r]^-{\sum{p^*_s}}&\H^1(\overline{X'})
}$$ is an isomorphism of $\Gamma_K$-modules.
\item For $i=2$,
$$\xymatrix{
\displaystyle\bigoplus_{s=1}^d\H^2(\overline{X}\nosp_s)\ar[r]^-{\sum{p^*_s}}&\H^2(\overline{X'})
}$$
is an injective homomorphism of $\Gamma_K$-modules. Moreover, as abelian groups $\bigoplus_{s=1}^d\H^2(\overline{X}\nosp_s)$ is a direct summand of $\H^2(\overline{X'})$ with complement the image of $\bigoplus_{1\leq s<t\leq d}\Hom(\overline{S}\nosp_s\otimes\overline{S}\nosp_t,\mu_n)$ by $\epsilon=\sum\epsilon_{s<t}$. \ \\
(cf. the paragraph before Proposition \ref{H12} for the definition of the map $\epsilon_{s<t}$ with notation $\epsilon_{12}$ instead.)
\end{enumerate}
\end{prop}

\begin{proof}
It follows from Proposition \ref{H12} by induction on the number of components.
\end{proof}

\begin{lem}\label{vanish}
With the notation above,
the composition
$$\xymatrix{
\Hom(\overline{S}\nosp_s\otimes\overline{S}\nosp_t,\mu_n)\ar[r]^-{\epsilon_{s<t}}&\H^2(\overline{X'})\ar[r]^-{\overline{f'}^*}&\H^2(\overline{T'})
}$$
vanishes.
\end{lem}

\begin{proof}
Recall that $\epsilon_{s<t}$ is defined to be $\epsilon_{s<t}(\psi)=\psi_*(p^*_s[\overline{T}\nosp_s]\cup p^*_t[\overline{T}\nosp_t])$ for any $\psi\in\Hom(\overline{S}\nosp_s\otimes\overline{S}\nosp_t,\mu_n)$. We are going to pull back the cup product to $\overline{T'}$.

The right square of the diagram below commutes by functoriality of cohomologies
$$\xymatrix{
\H^1(\overline{X}\nosp_s,\overline{S}\nosp_s)\times\H^1(\overline{X}\nosp_t,\overline{S}\nosp_t)\ar[rr]^-{\cup\circ(p^*_s,p^*_t)}\ar[d]^{(\overline{f}^*_s,\overline{f}^*_t)}&&\H^2(\overline{X'},\overline{S}\nosp_s\otimes\overline{S}\nosp_t)\ar[r]^-{\psi_*}\ar[d]^{\overline{f'}^*}&\H^2(\overline{X'})\ar[d]^{\overline{f'}^*}\\
\H^1(\overline{T}\nosp_s,\overline{S}\nosp_s)\times\H^1(\overline{T}\nosp_t,\overline{S}\nosp_t)\ar[rr]^-{\cup\circ(p^*_s,p^*_t)}&&\H^2(\overline{T'},\overline{S}\nosp_s\otimes\overline{S}\nosp_t)\ar[r]^-{\psi_*}&\H^2(\overline{T'}),
}$$
and the commutativity of the left square follows from the fact that $\overline{f'}^*$ commutes with cup products and $\overline{f}\nosp_s\circ p_s=p_s\circ\overline{f'}$ and $\overline{f}\nosp_t\circ p_t=p_t\circ\overline{f'}$.

Observe that the class $[\overline{T}\nosp_s]$ vanishes when being pulled back by $\overline{f}\nosp_s$ to $\H^1(\overline{T}\nosp_s,\overline{S}\nosp_s)$ (and the same for $t$), cf. \S \ref{torsorrecall}. Then we find that $\overline{f'}^*\circ\epsilon_{s<t}=0$.
\end{proof}

\subsection{Comparison of diverse cohomology groups}\label{comparisonsubsection}

Let $X$ be a geometrically connected $L$-variety and $X'=\R_{L/K}X$ be its Weil restriction. Consider the natural homomorphism between relevant terms of the spectral sequence (\ref{ss}) applied to $X$ and $X'$
$$\phi_{i,j}:\H^i(L,\H^j(\overline{X}))\to\H^i(K,\H^j(\overline{X'}))$$

\begin{lem}\label{lem1}
When $j=1$, then for any $i$,
$$\phi_{i,1}:\H^i(L,\H^1(\overline{X}))\to\H^i(K,\H^1(\overline{X'}))$$
is an isomorphism of abelian groups.
\end{lem}
\begin{proof}
It follows immediately from Shapiro's lemma \cite[\S I.5.8.b)]{SerreGaloisCoh} and the fact that $\H^1(\overline{X'})\simeq\bigoplus_{s=1}^d\H^1(\overline{X}\nosp_s)$ is  induced module $\Ind_{\Gamma_L}^{\Gamma_K}\H^i(\overline{X})$ by Proposition \ref{inducedmodule}(1).
\end{proof}

\begin{lem}\label{lem2}
Let $X$ be a geometrically connected $L$-variety and $X'=\R_{L/K}X$ be its Weil restriction. Suppose that $X'(\A_K)\neq\varnothing$
Then as subgroups of $\H^2(X')$ we have an equality
$$\phi(\H^2(X))+\im(\H^2(K))=\partial^{-1}_{X'}(\phi(\H^2(\overline{X})^{\Gamma_L})),$$
where $\partial_{X'}:\H^2(X')\to\H^2(\overline{X'})^{\Gamma_K}$ is the homomorphism given by the relevant spectral sequence.
\end{lem}

\begin{proof}
By hypothesis, both $X$ and $X'$ are geometrically connected and both $X(\A_L)$ and $X'(\A_K)$ are non-empty, we have the following commutative diagram with exact rows by Corollary \ref{sscor}(1)
$$\xymatrix{
\H^2(L)\ar[r]\ar[d]^\phi&\H_1^2(X)\ar[r]\ar[d]^\phi&\H^1(L,\H^1(\overline{X}))\ar[r]\ar[d]_\simeq^{\phi_{1,1}}&0\\
\H^2(K)\ar[r]&\H_1^2(X')\ar[r]&\H^1(K,\H^1(\overline{X'}))\ar[r]&0
}$$
where the vertical homomorphisms are naturally induced by Weil restriction, cf. \S \ref{notationsection}.
Here the vertical homomorphism $\phi_{1,1}$ is an isomorphism by Lemma \ref{lem1}. It follows that
\begin{equation}\label{H_1}
\H_1^2(X')=\phi(\H_1^2(X))+\im(\H^2(K)).
\end{equation}

Consider the following commutative diagram with exact rows by Corollary \ref{sscor}(2)
$$\xymatrix{
0\ar[r]&\H_1^2(X)\ar[r]\ar[d]^\phi&\H^2(X)\ar[r]^{\partial_X}\ar[d]^\phi&\H^2(\overline{X})^{\Gamma_L}\ar[r]\ar[d]^\phi&\H^2(L,\H^1(\overline{X}))\ar[d]^{\phi_{2,1}}_\simeq\\
0\ar[r]&\H_1^2(X')\ar[r]&\H^2(X')\ar[r]^{\partial_{X'}}&\H^2(\overline{X'})^{\Gamma_K}\ar[r]&\H^2(K,\H^1(\overline{X'}))
}$$
where the vertical homomorphism $\phi_{2,1}$ is an isomorphism by Lemma \ref{lem1}. Then the desired equality follows from diagram chasing and (\ref{H_1}).
\end{proof}

\begin{prop}\label{finalpropforunfair}
With the notation in \S \ref{universaltorsorsection}, for every class $[\sigma]\in\H^1(L,S)$
such that ${_\sigma\nosp T}'(\A_K)\neq\varnothing$
we have the  inclusion
\begin{equation}\label{finalinclusionforunfair}
{_\sigma\nosp f}'^*(\H^2(X'))\subset\phi(\H^2(_\sigma\nosp T))+\im(\H^2(K))
\end{equation}
as subgroups of $\H^2({_\sigma\nosp T}')$.
\end{prop}

\begin{proof}
Consider the commutative diagram
$$\xymatrix{
\H^2(X')\ar[r]^{{_\sigma\nosp f}'^*}\ar[d]^{\partial_{X'}} & H^2({_\sigma\nosp T}')\ar[d]^{\partial_{{_\sigma\nosp T}'}} & \H^2({_\sigma\nosp T})\ar[d]^{\partial_{_\sigma\nosp T}}\ar[l]_{\phi}\\
\H^2(\overline{X'})^{\Gamma_K}\ar[r]^{_\sigma\nosp{\overline{f'}}^*}&\H^2(\overline{_\sigma\nosp T'})^{\Gamma_K}&\H^2(\overline{_\sigma\nosp T})^{\Gamma_L}\ar[l]_{\overline{\phi}}
}$$
where vertical homomorphisms $\partial$ are  given by  relevant spectral sequences.
As mentioned in \S \ref{universaltorsorsection}, the universal $X'$-torsor of $n$-torsion ${_\sigma\nosp T}'$ is the Weil restriction of $_\sigma\nosp T$, it is geometrically connected.
Apply Lemma \ref{lem2} to those torsors  ${_\sigma\nosp T}'$ such that ${_\sigma\nosp T}'(\A_K)\neq\varnothing$,  diagram chasing implies that the desired inclusion (\ref{finalinclusionforunfair}) is a consequence of
\begin{equation}\label{inclusionoverclosedfield}
{_\sigma\nosp{\overline{f'}}^*}(\H^2(\overline{X'})^{\Gamma_K})\subset\overline{\phi}(H^2(\overline{_\sigma\nosp T})^{\Gamma_L})
\end{equation}
as subgroups of $\H^2(\overline{_\sigma\nosp T'})$.

As this last inclusion that we are going to prove concerns only cohomologies of varieties over the algebraic closure $\overline{L}$, twists by $\sigma$ do not affect anymore. We omit all $\sigma$ from now until the end of the proof.

Consider the following commutative diagram of $\Gamma_K$-modules with horizontal injective homomorphisms by Proposition \ref{inducedmodule} applied to $\overline{X'}$ and $\overline{T'}$ (the latter  is also connected cf. \S \ref{universaltorsorsection}).
$$\xymatrix{
\displaystyle\bigoplus_{s=1}^d\H^2(\overline{X}\nosp_s)\ar[r]^-{\alpha_X}\ar[d]^{\oplus{\overline{f}\nosp_s}\nosp^*}& \H^2(\overline{X'})\ar[d]^{\overline{f'}^*} \\
\displaystyle\bigoplus_{s=1}^d\H^2(\overline{T}\nosp_s)\ar[r]^-{\alpha_T}&\H^2(\overline{T'})
}$$
We claim that
\begin{equation}\label{f-alpha}
\im(\overline{f'}^*)\subset\im(\alpha_T)
\end{equation}
In fact, by Lemma \ref{vanish} the complement $\epsilon(\bigoplus_{1\leq s<t\leq d}\Hom(\overline{S}\nosp_s\otimes\overline{S}\nosp_t,\mu_n))$ of $\alpha_X(\bigoplus_{s=1}^d\H^2(\overline{X}\nosp_s))$ in $\H^2(\overline{X'})$ maps to $0$ by $\overline{f'}^*$, cf. Proposition \ref{inducedmodule}(2) for the structure of $\H^2(\overline{X'})$.
Then the claim follows from the diagram above.

The inclusion (\ref{f-alpha}) implies that
\begin{equation*}
\begin{array}{ll}
\displaystyle\overline{f'}^*(\H^2(\overline{X'})^{\Gamma_K})&\displaystyle\subset \im(\alpha_T)\cap\H^2(\overline{T'})^{\Gamma_K}\\
&\\
&\displaystyle=(\im(\alpha_T))^{\Gamma_K}\\
&\\
&\displaystyle=\alpha_T((\oplus_s\H^2(\overline{T}\nosp_s))^{\Gamma_K})\\
&\\
&\displaystyle=\alpha_T(\H^2(\overline{T})^{\Gamma_L})\\
&\\
&\displaystyle=\overline{\phi}(\H^2(\overline{T})^{\Gamma_L})
\end{array}
\end{equation*}
where the equalities follow respectively from the fact that
\begin{itemize}
\item[1st:]  $\im(\alpha_T)$ is a $\Gamma_K$-submodule of $\H^2(\overline{T'})$,
\item[2nd:]  $\alpha_T$ is an  injective homomorphism of $\Gamma_K$-modules,
\item[3rd:]  $\oplus_s\H^2(\overline{T}\nosp_s)=\Ind_{\Gamma_L}^{\Gamma_K}\H^2(\overline{T})$ is the induced module,
\item[4th:]  the restriction of $\alpha_T$ to the first component $\H^2(\overline{T})\to\H^2(\overline{T'})$ is nothing but the homomorphism $\overline{\phi}$.
\end{itemize}
This is nothing but the inclusion (\ref{inclusionoverclosedfield}) (with $\sigma$ omitted).
\end{proof}

\subsection{Proof of the unfair comparison (\ref{unfaireq})}\label{proofofunfairsubsection}

We are ready to complete the proof of the inclusion (\ref{unfaireq}) or its equivalent variant
$$\Phi\left(X(\A_L)^{\fab,\Br}\right)\subset X'(\A_K)^{\H^2(X')}.\leqno{(\ref{unfaireq}')}$$

\begin{proof}
We find that
\begin{equation*}
\begin{array}{ll}
\displaystyle X(\A_L)^{\fab,\Br}&\displaystyle\subset X(\A_L)^{f,\Br}\\
&\\
    &\displaystyle\subset\bigcup_{[\sigma]\in\H^1(L,S)}{_\sigma\nosp f}({_\sigma\nosp T}(\A_L)^{\H^2(_\sigma\nosp T)})\\
&\\
    &\displaystyle=\bigcup_{[\sigma]\in\H^1(K,S')}{_\sigma\nosp f}'({_\sigma\nosp T}'(\A_K)^{\phi(\H^2(_\sigma\nosp T))})\\
&\\
    &\displaystyle\subset X'(\A_K)^{\H^2(X')}

\end{array}
\end{equation*}
where
\begin{itemize}
\item[-] the first two inclusions are clear by definition;
\item[-] the equality follows from the compatibility of the \BM pairing mentioned in \S \ref{BMobsrecall};
\item[-] the last inclusion is a consequence of the functoriality (with respect to $_\sigma\nosp f'$) of the \BM pairing and Proposition \ref{finalpropforunfair}.
\end{itemize}
\end{proof}

\section{Proof of the main result}\label{mainproofsection}
We make use of the intermediate comparison (Proposition \ref{unfairprop}) by passing to the ``universal cover'' to prove the difficult inclusion of our main result Theorem \ref{mainthm}.

\subsection{Further properties of finite torsors}
We recall some further properties developed by M. Stoll of torsors under finite algebraic groups.
We call such torsors \emph{finite torsors} for short. In his paper  \cite{Stoll07}, varieties are usually supposed projective, but the proofs in \S 5 of his paper on finite torsors also work  without the projectiveness assumption.

Let $V$ be a geometrically connected variety defined over a number field $k$. Recall that a morphism between two $V$-torsors $W_1\buildrel{H_1}\over\to V$ and $W_2\buildrel{H_2}\over\to V$ is a $k$-morphism of algebraic groups $H_1\to H_2$ and an equivariant (with respect to the group actions) morphism of $V$-schemes $W_1\to W_2$. In such a situation, we say for simplicity that $W_1$ \emph{maps to} $W_2$ as $V$-torsors, and it does be the case on the level of cohomology sets $\H^1(V,H_1)\to\H^1(V,H_2)$.

With a connectedness assumption, geometrical morphisms between finite torsors come from morphisms defined over the ground field.

\begin{prop}[{\cite[Lemma 5.6]{Stoll07}}]\label{morfinitetorsor}
Let $V$ be a geometrically connected $k$-variety. Let $W_1\buildrel{H_2}\over\to V$ and $W_2\buildrel{H_2}\over\to V$ be two finite $V$-torsors. Assume that $W_1$ is geometrically connected. If $\overline{W}_1\buildrel{\overline{H}_1}\over\to \overline{V}$ maps to $\overline{W}_2\buildrel{\overline{H}_2}\over\to \overline{V}$ as $\overline{V}$-torsors, then there exists a class $[\tau]\in \H^1(k,H_2)$ such that $W_1\buildrel{H_2}\over\to V$ maps to the twist ${_\tau\nosp W}_2\buildrel{{_\tau\nosp H}_2}\over\to V$ as $V$-torsors.
\end{prop}

The notion of a cofinal family of coverings, which is in some sense the universal cover, is defined by Stoll as follows.

\begin{defn}
A family of finite $V$-torsors $(W_i\buildrel{H_i}\over\to V)_{i\in I}$ with $W_i$ geometrically connected over $k$ is called a \emph{cofinal family of coverings} of $V$ if for every $\overline{V}$-torsor $\overline{W}\buildrel{\overline{H}}\over\to\overline{V}$ over $\overline{k}$ with $\overline{W}$ connected (not necessarily defined over $k$) there exists a torsor $W_i\buildrel{H_i}\over\to V$ in the family such that it maps to $\overline{W}$ after being base extended to $\overline{k}$.
\end{defn}

Then Stoll obtained a sufficient condition for the existence of such a family.

\begin{lem}[{\cite[Lemma 5.7]{Stoll07}}]\label{cofinal}
Let $V$ be a geometrically connected $k$-variety. If $V(\A_k)^\et\neq\varnothing$, then there exists a cofinal family of coverings of $V$.
\end{lem}

\subsection{Proof of the main result}

\begin{proof}[Proof of Theorem \ref{mainthm}]
It remains to prove the difficult inclusion
\begin{equation}\label{hardhalf}
\Phi\left(X(\A_L)^{\et,\Br}\right)\subset X'(\A_K)^{\et,\Br}.
\end{equation}
We may assume that $X(\A_L)^{\et,\Br}\neq\varnothing$, a fortiori $X(\A_L)^\et\neq\varnothing$. There exists a cofinal family $(f_i:Y_i\buildrel{G_i}\over\to X)_{i\in I}$ of coverings of $X$ according to Lemma \ref{cofinal} applied to $X$ over $L$.

By passing to the prefixed algebraic closure of $L$, we see that the Weil restriction $(f'_i:Y'_i\buildrel{G'_i}\over\to X')_{i\in I}$ is a cofinal family of coverings of $X'$, see the proof of \cite[Proposition 5.15]{Stoll07} for details.

The right hand side of (\ref{hardhalf}) is the intersection of
$$X'(\A_K)^{\eta,\Br}=\displaystyle\bigcup_{[\sigma]\in \H^1(k,D)}{_\sigma\nosp \eta}({\nosp_\sigma\nosp \Delta}(\A_K)^\Br)$$
for all finite $X'$-torsors $\eta:\Delta\buildrel{D}\over\to X'$ under $D$ over $K$.
It suffices to show  that
\begin{equation}\label{aim}
\Phi\left(X(\A_L)^{\et,\Br}\right)\subset X'(\A_K)^{\eta,\Br}.
\end{equation}

Take base change to the prefixed algebraic closure of $K$ and  consider the $\overline{X}'$-torsor $\overline{\Delta}_0\buildrel{\overline{D}_0}\over\to\overline{X}'$ given by the connected component $\overline{\Delta}_0$ of $\overline{\Delta}$ and its stabilizer $\overline{D}_0$ in $\overline{D}$. By definition, there exists an $i\in I$ such that  $\overline{Y}'_i\buildrel{\overline{G}_i}\over\to \overline{X}'$ maps to $\overline{\Delta}_0\buildrel{\overline{D}_0}\over\to\overline{X}'$ as $\overline{X}'$-torsors and hence it maps to $\overline{\Delta}\buildrel{\overline{D}}\over\to\overline{X}'$ as $\overline{X}'$-torsors. According to Proposition \ref{morfinitetorsor}, there exists a class $[\tau]\in\H^1(K,D)$ such that $f'_i:Y'_i\buildrel{G'_i}\over\to X'$ maps to the twist ${_\tau\nosp\eta}:{_\tau\nosp \Delta}\buildrel{{_\tau\nosp D}}\over\to X'$ as $X'$-torsors.
For each class $[\sigma]\in\H^1(K,G'_i)$, we denote by $[\tilde{\sigma}]\in\H^1(K,{_\tau\nosp D})$ its image under the map between cohomology sets induced by $G'_i\to{_\tau\nosp D}$. Then ${_\sigma\nosp f'_i}:{_\sigma\nosp Y'_i}\buildrel{{_\sigma\nosp G'_i}}\over\To X'$ maps to ${_{\tilde{\sigma}}\nosp({_\tau\nosp\eta})}:{_{\tilde{\sigma}}\nosp({_\tau\nosp\Delta})}\buildrel{{_{\tilde{\sigma}}\nosp({_\tau\nosp D})}}\over\To X'$ as $X'$-torsors, i.e. ${_\sigma\nosp f'_i}={_{\tilde{\sigma}}\nosp({_\tau\nosp\eta})}\circ g$ for a certain morphism of $X'$-torsors $g:{_\sigma\nosp Y'_i}\to{_{\tilde{\sigma}}\nosp({_\tau\nosp\Delta})}$ which depends on $\sigma$.
By functoriality of the \BM pairing, we obtain the following inclusion
$${_\sigma\nosp f'_i}({_\sigma\nosp Y'_i}(\A_K)^\Br)={_{\tilde{\sigma}}\nosp({_\tau\nosp\eta})}\circ g(Y'_i(\A_K)^\Br)\subset{_{\tilde{\sigma}}\nosp({_\tau\nosp\eta})}({_{\tilde{\sigma}}\nosp({_\tau\nosp\Delta})}(\A_K)^\Br).$$
Taking union over all classes $[\sigma]\in\H^1(K,G'_i)$, we obtain
\begin{equation}\label{cofinalcontrol}
\begin{array}{ll}
\displaystyle X'(\A_K)^{f'_i,\Br}&\displaystyle=\bigcup_{[\sigma]\in\H^1(K,G'_i)}{_\sigma\nosp f'_i}({_\sigma\nosp Y'_i}(\A_K)^\Br)\\
&\\
\displaystyle&\displaystyle\subset\bigcup_{[\sigma]\in\H^1(K,G'_i)}{_{\tilde{\sigma}}\nosp({_\tau\nosp\eta})}({_{\tilde{\sigma}}\nosp({_\tau\nosp\Delta})}(\A_K)^\Br)\\
&\\
\displaystyle&\displaystyle\subset\bigcup_{[\delta]\in\H^1(K,D)}{_\delta\nosp\eta}({_\delta\nosp\Delta})(\A_K)^\Br)\\
&\\
\displaystyle&\displaystyle= X'(\A_K)^{\eta,\Br}.
\end{array}
\end{equation}

According to a result of the first  author \cite[Th\'eor\`eme 1.1]{Cao2017}, we have
\begin{equation}\label{caoeq}
X(\A_L)^{\et,\Br}=\bigcup_{[\sigma]\in\H^1(L,G_i)}{_\sigma\nosp f_i}({_\sigma\nosp Y_i}(\A_L)^{\et,\Br}).
\end{equation}
Since $Y_i$ is an element in the cofinal family of coverings of $X'$, it is geometrically connected and so is ${_\sigma\nosp Y_i}$.
Proposition \ref{unfairprop} (the unfair comparison) applied to  ${_\sigma\nosp Y_i}$ and its Weil restriction ${_\sigma\nosp Y'_i}$ implies that
$$\Phi\left({_\sigma\nosp Y_i}(\A_L)^{\et,\Br}\right)\subset\Phi\left({_\sigma\nosp Y_i}(\A_L)^{\fab,\Br}\right)\subset{_\sigma\nosp Y'_i}(\A_K)^\Br.$$
Applying $\Phi$ to the equation (\ref{caoeq}) and taking union over all classes $[\sigma]\in\H^1(L,G_i)=\H^1(K,G'_i)$, it turns out that
\begin{equation}\label{finalstep}
\Phi\left(X(\A_L)^{\et,\Br}\right)\subset\bigcup_{[\sigma]\in\H^1(K,G'_i)}{_\sigma\nosp f'_i}({_\sigma\nosp Y'_i}(\A_K)^{\Br})=X'(\A_K)^{f'_i,\Br}.
\end{equation}
The inclusion (\ref{aim}) now follows from (\ref{cofinalcontrol}) and (\ref{finalstep}), which completes the proof.
\end{proof}

\bigskip

\begin{footnotesize}
\noindent\textbf{Acknowledgements.}
The first author is partially supported by CAS Project for Young Scientists in Basic Research (Grant No. YSBR-032). The second author is partially supported by
National Natural Science Foundation of China (Grant No.12071448).
Both authors are also partially supported by Innovation Program for Quantum Science and Technology (Grant No. 2021ZD0302904).
\end{footnotesize}

\bigskip


\bigskip

\bibliographystyle{alpha}
\bibliography{mybib1}

\end{document}